\documentclass[12pt]{amsart}
\usepackage{amsmath}
\usepackage{amsfonts}
\usepackage{amssymb}
\usepackage{amsthm}
\usepackage{eucal}
\usepackage{multicol}
\newtheorem{theorem}{Theorem}[section]

\usepackage{mathrsfs}
\usepackage{eucal}
\usepackage{graphicx}

\newcommand{\pp}{\mathbb{P}}

\renewcommand{\O}{\mathscr{O}}

\newcommand{\ff}{\mathbb{F}}

\usepackage{dutchcal}

\DeclareMathOperator{\codim}{codim}
\newtheorem{Lemma}[theorem]{Lemma}
\newtheorem{prop}[theorem]{Proposition}
\newtheorem{Corollary}[theorem]{Corollary}

\DeclareMathOperator{\Pic}{Pic}

\DeclareMathOperator{\Sym}{Sym}

\newcommand{\ee}{{\vec{e}} \nobreak\hspace{.16667em plus .08333em}'}
\newcommand{\Sigmabar}{\overline{\Sigma}}
\newcommand{\e}{\vec{e}}
\usepackage{stmaryrd}

\theoremstyle{definition}
\newtheorem{Example}[theorem]{Example}

\theoremstyle{remark}

\theoremstyle{remark}
\newtheorem*{warning}{Warning}
\usepackage{tikz}
\usepackage{tikz-cd}
\usepackage{wasysym, stackengine, makebox, graphicx}

\tikzset{cong/.style={draw=none,edge node={node [sloped, allow upside down, auto=false]{$\cong$}}},
         Isom/.style={draw=none,every to/.append style={edge node={node [sloped, allow upside down, auto=false]{$\cong$}}}}}

\usetikzlibrary{matrix,arrows,decorations.pathmorphing}
\usepackage[margin=1.2in]{geometry}

\numberwithin{equation}{section}

\title{Refined Brill-Noether theory for all trigonal curves}
\date{\today}
\author{Hannah K. Larson}
\begin{document}

\maketitle

\begin{abstract}
Trigonal curves provide an example of Brill-Noether special curves. Theorem 1.3 of \cite{L} characterizes the Brill-Noether theory of general trigonal curves and the refined stratification by Brill-Noether splitting loci, which parametrize line bundles whose push forward to $\pp^1$ has a specified splitting type. This note describes the
refined stratification for \textit{all} trigonal curves. Given the Maroni invariant of a trigonal curve, we determine the dimensions of all Brill-Noether splitting loci and describe their irreducible components. When the dimension is positive, these loci are connected, and if furthermore the Maroni invariant is $0$ or $1$, they are irreducible.
\end{abstract}

\section{Introduction}

The celebrated Brill-Noether theorem describes the degree $d$ maps of general curves to $\pp^r$. Such maps correspond to line bundles in the \textit{Brill-Noether locus}
\[W^r_d(C) := \{L \in \Pic^d(C) : h^0(L) \geq r+ 1\}.\]
When $C$ is a general curve of genus $g$, the Brill-Noether theorem says that
\begin{enumerate}
\item $\dim W^r_d(C) = \rho(g, r, d) := g - (r + 1)(d - g + r)$ \cite{GH},
\item $W^r_d(C)$ is smooth away from $W^{r+1}_d(C)$ \cite{GP},
\item $W^r_d(C)$ is irreducible when $\rho(g, r, d) > 0$ \cite{FL}.
\end{enumerate}

The Brill-Noether theorem fails for special curves, notably curves of low gonality (see \cite{CPJ, CM1,CM2, L, JR, Pf}).
As a first example, if $C$ is a general trigonal curve of genus $6$, then $C \subset \pp^1 \times \pp^1$ as a curve of bidegree $(3, 4)$ and $W^1_4(C)$ consists of two components of different dimensions: the isolated $g^1_4$ and the $g^1_3$ plus a base point. 
More generally, if $C$ is any trigonal curve of genus $g$, we shall see (in Lemma \ref{ez}) that for $r \geq 1$, $W^r_d(C)$ consists of (at most) two components: one of dimension $\rho(g, r, d) + (r - 1)(g - d + r - 1)$ and a second of dimension $\rho(g, r, d) + r(g - d + r - 2)$ whenever $r \geq d - g + 2$ (see Figure 1).
\begin{center}
Figure 1: Codimensions of Brill-Noether loci

\begin{tikzpicture}
\node at (1,2.3) {\tiny $r+1$};
\draw[<->] (0, 2) -- (2,2);
\draw [<->] (-.5, 0) -- (-.5,1.5);
\node [rotate = 90] at (-.8, .75) {\tiny $g-d+r$};
\draw[fill=blue!50] (0, 0) rectangle (.5, .5);
\draw[fill=blue!50] (.5, 0) rectangle (1, .5);
\draw[fill=blue!50] (1, 0) rectangle (1.5, .5);
\draw[fill=blue!50] (1.5, 0) rectangle (2, .5);
\draw[fill=blue!50] (0, .5) rectangle (.5, 1);
\draw[fill=blue!50] (.5, .5) rectangle (1, 1);
\draw[fill=blue!50] (1, .5) rectangle (1.5, 1);
\draw[fill=blue!50] (1.5, .5) rectangle (2, 1);
\draw[fill=blue!50] (0, 1) rectangle (.5, 1.5);
\draw[fill=blue!50] (.5, 1) rectangle (1, 1.5);
\draw[fill=blue!50] (1, 1) rectangle (1.5, 1.5);
\draw[fill=blue!50] (1.5, 1) rectangle (2, 1.5);
\end{tikzpicture}
\hspace{1.2in}
\begin{tikzpicture}
\draw[fill=blue!50] (0, 0) rectangle (.5, .5);
\draw[fill=blue!50] (.5, 0) rectangle (1, .5);
\draw (1, 0) rectangle (1.5, .5);
\draw (1.5, 0) rectangle (2, .5);
\draw[fill=blue!50] (0, .5) rectangle (.5, 1);
\draw[fill=blue!50] (.5, .5) rectangle (1, 1);
\draw (1, .5) rectangle (1.5, 1);
\draw (1.5, .5) rectangle (2, 1);
\draw[fill=blue!50] (0, 1) rectangle (.5, 1.5);
\draw[fill=blue!50] (.5, 1) rectangle (1, 1.5);
\draw[fill=blue!50] (1, 1) rectangle (1.5, 1.5);
\draw[fill=blue!50] (1.5, 1) rectangle (2, 1.5);
\end{tikzpicture}
\hspace{1in}
\begin{tikzpicture}
\draw[fill=blue!50] (0, 0) rectangle (.5, .5);
\draw (.5, 0) rectangle (1, .5);
\draw (1, 0) rectangle (1.5, .5);
\draw (1.5, 0) rectangle (2, .5);
\draw[fill=blue!50] (0, .5) rectangle (.5, 1);
\draw[fill=blue!50] (.5, .5) rectangle (1, 1);
\draw[fill=blue!50] (1, .5) rectangle (1.5, 1);
\draw[fill=blue!50] (1.5, .5) rectangle (2, 1);
\draw[fill=blue!50] (0, 1) rectangle (.5, 1.5);
\draw[fill=blue!50] (.5, 1) rectangle (1, 1.5);
\draw[fill=blue!50] (1, 1) rectangle (1.5, 1.5);
\draw[fill=blue!50] (1.5, 1) rectangle (2, 1.5);
\end{tikzpicture}

\hspace{.2in} generic $\codim W^r_d(C)$ 
\hspace{.5in} $\codim$ of components of $W^r_d(C)$ for $C$ trigonal
\vspace{.1in}
\end{center}

 A key insight of \cite{L} is that these different components are explained by a refined stratification of $\Pic^d(C)$ that keeps track not just of $h^0(C, L)$, but of the splitting type of the push forward of $L$ to $\pp^1$. Given a degree $k$ cover $f: C \rightarrow \pp^1$ and $L \in \Pic^d(C)$, the push forward $f_*L$ is a rank $k$ vector bundle on $\pp^1$. Thus, $f_*L \cong \O_{\pp^1}(e_1) \oplus \cdots \oplus \O_{\pp^1}(e_k)$ for some $e_1 \leq \cdots \leq e_k$. Given $\vec{e} = (e_1, \ldots, e_k)$, we abbreviate the corresponding sum of line bundles by $\O(\vec{e})$. By Riemann-Roch, 
 \begin{equation} \label{chi}
 e_1 + \ldots + e_k = \deg(f_*L) = \chi(f_*L) - k = \chi(L) - k = d - g + 1 - k.
 \end{equation}
\textit{Brill-Noether splitting loci} are defined as
\[\Sigma_{\vec{e}}(C, f) := \{L \in \Pic^d(C) : f_*L \cong \O(\vec{e}) \}.\]
The \textit{expected codimension} of $\Sigma_{\vec{e}}(C, f)$ is defined as
\[u(\vec{e}) := \sum_{j > i} \max\{0, e_j - e_i\}.\]
We define a partial ordering on splitting types by $\ee \leq \e$ if $\ee$ is a specialization of $\e$ in the moduli space of vector bundles on $\pp^1$ bundles, i.e. if $e_1'+\ldots + e_j' \leq e_1 + \ldots + e_j$ for all $j$.	
We then define closed subvarieties
\[\Sigmabar_{\e}(C, f) := \bigcup_{\ee \leq \e} \Sigma_{\ee}(C, f).\]
\begin{warning}
For general $f: C \rightarrow \pp^1$, $\Sigmabar_{\e}(C, f)$ is the closure of $\Sigma_{\e}(C, f)$ by \cite[Lemma 2.1]{L}. However, for special $C$, we may have $\Sigmabar_{\e}(C, f)$ non-empty even when $\Sigma_{\vec{e}}(C, f)$ is empty (see Example \ref{ex1}).
\end{warning} 

The two components of $W^r_d(C)$ of a trigonal curve pictured in Figure 1 correspond to two different splitting loci: $\Sigmabar_{(d-g-1-r, \lfloor \frac{r - 1}{2} \rfloor, \lceil \frac{r - 1}{2} \rceil)}(C, f)$ and $\Sigmabar_{(\lfloor \frac{d-g-2-r}{2}\rfloor, \lceil \frac{d-g-2-r}{2}\rceil, r)}(C, f)$, both of which occur in the expected codimension.
More generally, \cite[Theorem 1.2]{L} proves that for \textit{general} $k$-gonal curves, all Brill-Noether splitting loci are smooth of the expected dimension. It remains an open question if these splitting loci are irreducible (see the remark following \cite[Theorem 1.2]{L} and \cite[Conjecture 1.2]{CPJ}).

Here, we answer this question for trigonal curves (and make precise what ``general" means) as part of a study of the refined stratification for \textit{all} trigonal curves.
It turns out that the geometry of the refined stratification for any trigonal curve is governed by the Maroni invariant of the curve.
Recall that the \textit{Maroni invariant} of a trigonal curve $f: C \rightarrow \pp^1$ is the non-negative integer $n$ such that $f_*\O_C \cong \O_{\pp^1}(-\frac{g+n}{2} - 1) \oplus \O_{\pp^1}(-\frac{g-n}{2} -1) \oplus \O_{\pp^1}$. In this case, $C$ has a canonical map to the Hirzebruch surface $\ff_n := \pp(\O_{\pp^1} \oplus \O_{\pp^1}(n))$. We give a formula for the dimensions of all splitting loci, depending on the Maroni invariant. We also determine when they are irreducible and describe their irreducible components when they are reducible.

\begin{theorem} \label{maint}
Suppose $f: C \rightarrow \pp^1$ is a trigonal curve with Maroni invariant $n$ over an algebraically closed field of characteristic zero. Let $a \leq b \leq c$ be integers.
If $b - a \leq 1$ or $c - b \leq 1$, then $\Sigmabar_{(a, b, c)}(C, f)$ is irreducible of the expected dimension $g - u(a, b, c)$.
When $b - a \geq 2$ and $c - b \geq 2$, the following hold.
\begin{enumerate}
\item \label{s1} If $\frac{g + n}{2} + 1 - c + a \geq n$, then
$\Sigma_{(a, b, c)}(C, f)$ is irreducible of the expected dimension $g - u(a, b, c)$.
\item Suppose $1 \leq \frac{g + n}{2} +1 - c + a < n$. If $c - b, b - a \leq 1 + \frac{g - n}{2}$, then $\Sigma_{(a, b, c)}(C, f)$ is connected of dimension $\frac{g + n}{2} + 1 - c + a$, with irreducible components geometrically described in Proposition \ref{smalln}. If $c - b$ or $b - a > 1 + \frac{g - n}{2}$, then $\Sigma_{(a, b, c)}(C, f)$ is empty.
\item \label{num} If $\frac{g + n}{2} + 1 - c + a = 0$ and $C$ is a general curve of Maroni invariant $n$, then $\Sigma_{(a, b, c)}(C, f)$ consists of ${c - a - 2n \choose b - a - n}$ points.
\end{enumerate}
In particular, when $f: C \rightarrow \pp^1$ has Maroni invariant $0$ or $1$, all splitting loci have the expected dimension, and they are irreducible when their dimension is positive.
\end{theorem}

\begin{Example}[$g=11, n=3, d = 0$] \label{ex1}
Suppose $C \subset \ff_3$ is a genus $11$ curve of Maroni invariant $3$, so $f_*\O_C = \O_{\pp^1}(-8) \oplus \O_{\pp^1}(-5) \oplus \O_{\pp^1}$. 
In the diagram below, the arrows give the partial ordering.
Light grey indicates the expected dimensions of pieces in the stratification that are replaced with the red shape for this Maroni invariant.

Note that $\Sigmabar_{(-7,-6,0)}$ is irreducible of dimension $0$, so $\{\O_C\} = \Sigma_{(-8,-5, 0)} \subset \Sigmabar_{(-7,-6,0)}$ implies $\Sigmabar_{(-7, -6, 0)} = \{\O_C\}$. Hence, the open stratum $\Sigma_{(-7, -6, 0)}$ is empty. Although the expected dimension of $\Sigma_{(-8, -4, -1)}$ is $0$, Theorem \ref{maint} (2) shows $\dim \Sigma_{(-8, -4, -1)} = 1$.
This curve has three components, described by Example \ref{631}. The surface $\Sigma_{(-7, -5, -1)}$ also has three components; it is colored blue to indicate that it is the expected dimension but reducible.
 The remaining splitting loci, drawn in black below or not pictured, are all irreducible of the expected dimension.

\begin{center}
\begin{tikzpicture}[scale=.65, every node/.style={scale=0.65}]
\node at (-9.5, 3) {dimension};
\node at (0, 3) {splitting type};
\node at (-9.5, 2) {$-2$};
\node at (-9.5, 1) {$-1$};
\node at (-9.5, 0) {$0$};
\node at (-9.5, -1) {$1$};
\node at (-9.5, -2) {$2$};
\node at (-9.5, -3) {$3$};
\node at (-9.5, -4) {$4$};
\node at (-9.5, -5) {$5$};
\node at (-9.5, -6) {$6$};
\node at (-9.5, -7) {$\vdots$};
\node at (0, -7) {$\vdots$};
\node at (10, -6) {$\vdots$};

\node at (-1.5, 2) {\color{black!60} $\varnothing =$};
\node at (0, 2) {\color{black!60} $(-8, -5, 0)$};
\node at (0, 0) {\color{red} $\mathbf{(-8, -5, 0})$};
\node at (-5, 0) {\color{black!60} $(-7, -6, 0)$};
\node at (-6.5, 0) {\color{red} $\pmb{\varnothing =}$};
\node at (5, 0) {\color{black!60} $(-8, -4, -1)$};
\node at (0, -2) {\color{blue} $\mathbf{(-7, -5, -1)}$};
\node at (5, -1) {\color{red} $\mathbf{(-8, -4, -1)}$};
\node at (-5, -3) {$\mathbf{(-6, -6, -1)}$};
\node at (0,-6) {$\mathbf{(-6,-5,-2)}$};
\node at (5,-4) {$\mathbf{(-7,-4,-2)}$};
\node at (10,-5) {$\mathbf{(-7, -3, -3)}$};
\node at (10, -2) {$\mathbf{(-8, -3, -2)}$};
\draw[->, color=black!40] (-1.2, 1.8) -- (-5+1.2, 0+.2);
\draw[<-, color=black!40] (-1.2, -1.8) -- (-5+1.2, 0-.2);
\draw[->, color=black!40] (1.2, 1.8) -- (5-1.2, 0+.2);
\draw[<-, color=black!40] (1.2, -1.8) -- (5-1.2, 0-.2);
\draw[->, color=black!40] (1.2+5, 1.8-2) -- (5-1.2+5, 0+.2-2);
\draw[->, color=black] (1.2, 1.8-2-2) -- (5-1.2, 0+.2-2-2);
\draw[->, color=black] (1.2+5, 1.8-2-2-2+.1) -- (5-1.2+5, 0+.2-2-2-1);

\draw[->, color=black] (1.2-5, 1.8-2-2-2+1) -- (5-1.2-5, 0+.2-2-2-1-1);

\draw[->, color=black] (-1.2, 1.8-4+.1) -- (-5+1.2, 1+.2-4-.1);

\draw[->, color=black] (-1.2+5, 1.8-4-2) -- (1.2, 0+.2-4-2);
\draw[->, color=black] (-1.2+5+5, 1.8-4) -- (1.2+5, 0+.2-4);

\draw[->, line width=.2mm, color=red] (1.2,-.1) -- (5-1.2, -1+.2);
\draw[<-, line width=.2mm, color=red] (1.2, -2+.1) -- (5-1.2, -1-.1);
\draw[->, line width=.2mm, color=red] (5+1.2, -1-.1) -- (10-1.2, -2+.1);

\end{tikzpicture}
\end{center}
\end{Example}

As a corollary of Theorem \ref{maint} \eqref{num}, we calculate the expected classes of all splitting degeneracy loci. (In \cite[\S 5]{L}, these classes were shown to be non-zero but not explicitly determined.)
\begin{Corollary} \label{cor}
Suppose $f: C \rightarrow \pp^1$ is a trigonal curve and $\Sigmabar_{(a, b, c)}(C, f)$ occurs in the expected codimension. Then its class in the Jacobian of $C$ is
\[[\Sigmabar_{(a, b, c)}(C, f)] = \begin{cases} \frac{1}{u(a, b, c)!} \theta^{u(a, b, c)} & \text{if $b - a \leq 1$ or $c - b \leq 1$} \\ \frac{1}{u(a,b,c)!} {c - a - 2 \choose b - a - 1} \theta^{u(a, b, c)} &\text{otherwise,} \end{cases}\]
where $\theta$ is the class of the theta divisor.
\end{Corollary}

\subsection*{Acknowledgements} Many key ideas of this paper were discovered with Izzet Coskun during a visit at University of Illinois -- Chicago. I am grateful for his insights and interest in discussing this problem.
I would also like to thank Ravi Vakil for helpful conversations; and David Eisenbud for asking if \cite[Example 1.1]{L} held for all genus $5$ trigonal curves, which was the initial motivation for this work.
I am also grateful to the Hertz Foundation Graduate Fellowship, NSF Graduate Research Fellowship, Maryam Mirzakhani Graduate Fellowship, and the Stanford Graduate Fellowship for their generous support.



\section{Descriptions of splitting loci}
Fix some trigonal curve $f: C \rightarrow \pp^1$ and let $H = f^*\O_{\pp^1}(1)$ be the fiber class.
When there is no confusion, we will write $\Sigma_{(a, b, c)}$ for $\Sigma_{(a, b, c,)}(C, f)$.
The
rank three splitting types $(a, b, c)$ with $a \leq b \leq c$ come in three flavors: 
\begin{enumerate}
\item[(I)] \label{1} If $b - a \leq 1$, the splitting locus is determined by a single rank condition 
\[\Sigmabar_{(a, b, c)} = \{L \in \Pic^d(C) : h^0(C, L(-cH)) \geq 1\}\] 
\item[(II)]  \label{2} If $c - b \leq 1$, the splitting locus is also determined by a single rank condition 
\[\Sigmabar_{(a, b, c)} = \{L \in \Pic^d(C) : h^1(C, L((-a - 2)H) \geq 1\}\]
\item[(III)] \label{3} If $b - a \geq 2$ and $c - b \geq 2$, two rank conditions are needed to determine the splitting locus
\[\Sigmabar_{(a, b, c)} = \{L \in \Pic^d(C) : h^0(C, L(-cH)) \geq 1 \text{ and } h^1(C, L((-a - 2)H) \geq 1\}\]
\end{enumerate}
Types (I) and (II) correspond to components of Brill-Noether loci, whose codimensions are pictured in Figure 1 of the introduction. It turns out these splitting loci are always irreducible of the expected dimension.

\begin{Lemma} \label{ez}
If $a - b \leq 1$ or $c - b \leq 1$, then $\Sigmabar_{(a,b,c)}(C, f)$ is irreducible of the expected dimension for every $f:C \rightarrow \pp^1$.
\end{Lemma}
\begin{proof}
We may rewrite (I), as saying $\Sigmabar_{a, b, c} = \{\O(cH + D): D \in W^0_{d - 3c}(C)\}$. Using \eqref{chi}, note that $d = a + b + c + g + 2$, so
\[d - 3c = g + a + b -  2c + 2 = g - u(a, b,c).\]
Considering the map $\mathrm{Sym}^{d-3c}(C) \rightarrow W^0_{d - 3c}(C)$ shows that $W^0_{d-3c}(C)$ --- and hence $\Sigmabar_{(a, b, c)}$ --- is irreducible of dimension $g - u(a, b, c)$.
Similarly, if $c - b \leq 1$, using Serre duality, we may rewrite (II) as saying 
\[\Sigmabar_{(a, b, c)} = \{\O(K+(a+2)H - D'): D' \in W^0_{g-u(a, b, c)}(C)\}.\]
so again this splitting locus is always irreducible of the expected dimension.
\end{proof}

The loci of splitting loci of type (III) are more subtle and will be the focus of the remainder of the paper.
These additional pieces of the stratification correspond to intersections of various components of Brill-Noether loci. We shall see in the next section that their geometry depends on the Maroni invariant of the curve.

The following gives a description of $\Sigma := \Sigma_{(a, b, c)}(C, f)$ in terms of subdivisors of a certain subcanonical linear series. 

\begin{Lemma} \label{point}
Suppose $a \leq b \leq c$ with $b - a, c - b \geq 2$. Then
\[\Sigma = \Sigma_{(a, b, c)} = \{\O(D + cH): h^0(C, \O(D)) = 1, h^0(C, K( - (c - a - 2) H - D)) = 1\}.\]
where $\deg D = g + 2 + a + b - 2c$.
\end{Lemma}
\begin{proof}
Suppose $L \in \Sigma$. 
The condition $1 = h^0(\pp^1, (f_*L)(-c)) = h^0(C, L(-cH))$ is equivalent to $L = \O(D + cH)$ for some effective divisor $D$ with $h^0(C, \O(D)) = 1$ and
\[\deg D = \deg L - 3c = g+2 + a+ b - 2c.\]
The other defining condition of $\Sigma_{(a, b, c)}$ gives
\begin{align*}
1 &= h^1(C, (f_*L)(-a - 2)) = h^1(C, L((-a - 2)H))  \\
&= h^0(C, K( - L + (a + 2)H)) = h^0(K(-(D + cH) + (a + 2) H)). \qedhere
\end{align*}
\end{proof}

To study the geometry of $\Sigma$, we introduce the variety
\begin{equation} \label{phi}
\Phi := \{(D, D') : D + D' \sim K - (c - a - 2)H\} \subset \Sym^{g+2 + a + b - 2c}(C) \times \Sym^{g + 2 + 2a - b - c}(C) 
\end{equation}
and the open subvariety
\begin{equation} \label{phic}
\Phi^\circ =\{(D, D') \in \Phi : h^0(C, \O(D)) = h^0(C, \O(D')) = 1\}.
\end{equation}
By Lemma \ref{point}, $\Sigma$ is isomorphic to $\Phi^\circ$ under the map $(D, D') \mapsto \O(D + cH)$. 

The key to understanding $\Phi$ is to realize $K - (c - a - 2)H$ as the restriction of a linear series on a Hirzebruch surface containing $C$.
We do this in Lemma \ref{reslem}.
This allows us to compute $\dim \Phi = h^0(C, K - (c - a -2)H) - 1$ in Lemma \ref{dim}.
Then in Lemma \ref{bign} and Proposition \ref{smalln}, we determine the conditions on $(a, b, c)$ such that $\Phi^\circ$ is non-empty and describe its irreducible components.


\section{Trigonal curves on Hirzebruch surfaces}

Suppose $f: C \rightarrow \pp^1$ is a smooth trigonal curve. Let $\iota: C \rightarrow \pp^{g-1}$ denote the canoical embedding. For each $p \in \pp^1$,
Geometric Riemann-Roch (see e.g. \cite[p.~12]{ACGH}) shows that $\iota(f^{-1}(p))$ is contained in a line. It follows that the image of $(f \times \iota): C \rightarrow \pp^1 \times \pp^{g-1}$ lies inside a $\pp^1$ bundle over $\pp^1$. This surface has the form $\ff_n := \pp(\O_{\pp^1} \oplus \O_{\pp^1}(n))$ for some $n \geq 0$, which is defined to be the \textit{Maroni invariant} of $f: C \rightarrow \pp^1$. We denote by $\pi$ the projection $\ff_n \rightarrow \pp^1$.

The Maroni invariant can also be described as follows: $f$ gives rise to an exact sequence
\[0 \rightarrow V \rightarrow f_* \omega_C \rightarrow \omega_{\pp^1} \rightarrow 0\]
where $V$ is a rank $2$, degree $g -1$ vector bundle. We may write $V = \O_{
\pp^1}\left(\frac{g-n}{2} - 1\right) \oplus \O_{\pp^1}\left(\frac{g + n}{2} - 1\right)$ for some $n$, which is the Maroni invariant.
The composition $f^*V \rightarrow f^*f_*\omega_C \rightarrow \omega_C$ induces a map $C \rightarrow \pp V \cong \ff_n$ factoring both $\iota$ and $f$. Explicitly, the line bundle $\O_{\pp V}(1)$ maps $\pp V$ into $\pp^{g-1}$ as a rational scroll, obtained as the union of lines joining corresponding points of rational normal curves of degrees $\frac{g - n}{2} - 1$ and $\frac{g+n}{2} - 1$ in complementary linear spaces.

\begin{multicols}{2}
\begin{center}
\begin{tikzcd}
&&& \pp^{g-1} \\
&C \arrow{r} \arrow[bend left = 20]{urr}{\iota} \arrow{dr}[swap]{f} & \pp V \arrow{d}{\pi} \arrow{ur}[swap]{|\O_{\pp V}(1)|} & \\
&& \pp^1
\end{tikzcd}
\end{center}

\begin{center}
\includegraphics[width=2.1in]{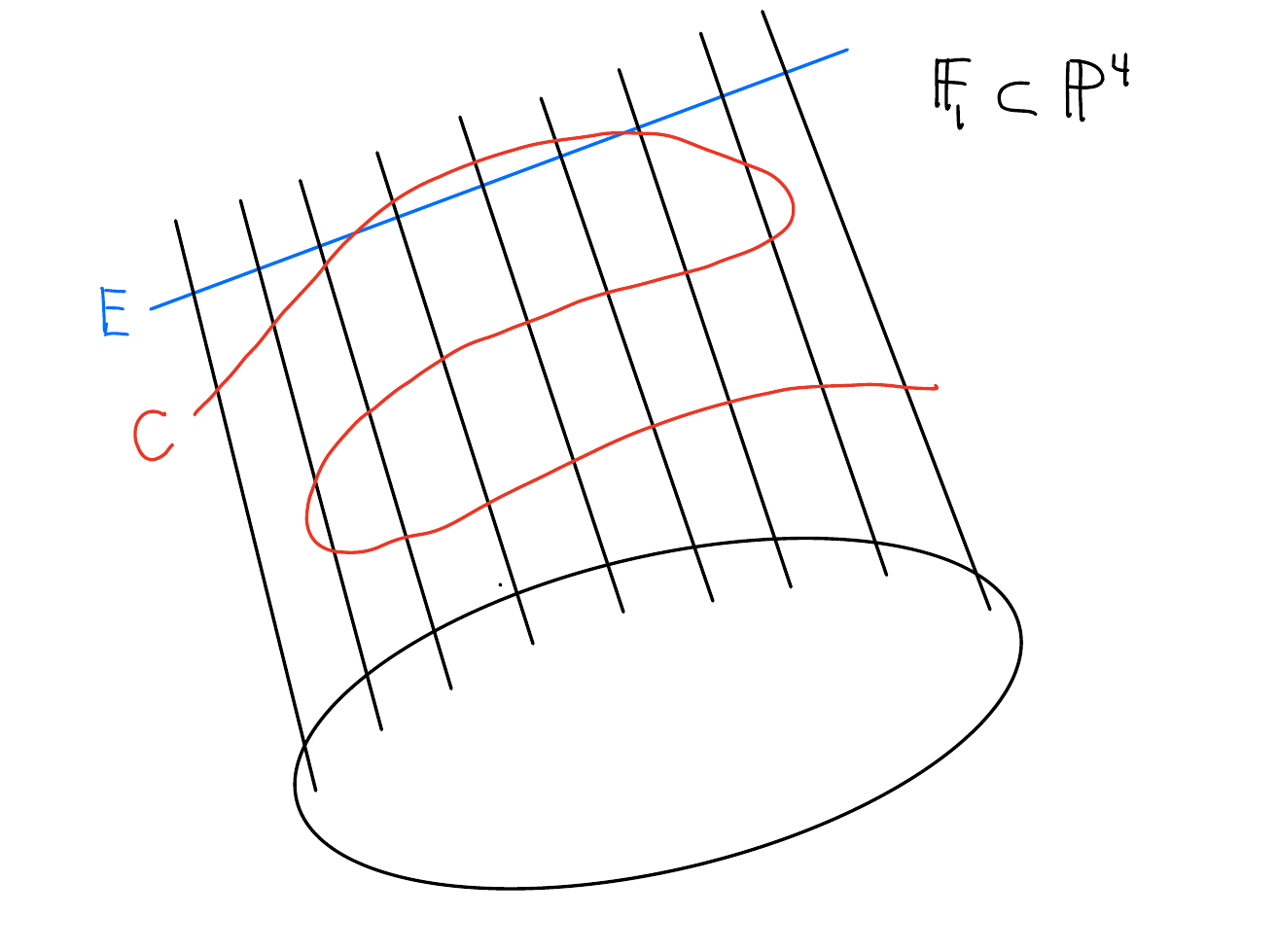}
\end{center}
 \end{multicols}
 
We recall some basic facts about divisors on Hirzebruch surfaces.
Let $E \subset \ff_n$ be the previously mentioned rational curve of degree $\frac{g - n}{2} - 1$, called the \textit{directrix}. Note that $\pi_*\O_{\ff_n}(E) = \O_{\pp^1} \oplus \O_{\pp^1}(-n)$.
Let $F \subset \ff_n$ be $\pi^{-1}(p)$ for a general point $p \in \pp^1$. The classes of $E$ and $F$ generate $A^*(\ff_n)$. By slight abuse of notation, we use the same letters for the classes of these curves. The intersection numbers are
\[E^2 = -n \qquad E.F = 1 \qquad F^2 = 0.\]
If $n > 0$, then the directrix $E$ is also characterized as the unique curve of self-intersection $-n$.

Throughout this section, we fix $a \leq b \leq c$ with $b - a, c - b \geq 2$.
The following lemma expresses the linear system $K - (c - a - 2)H$ that we are interested in as a restriction of a certain combination of $E$ and $F$.

\begin{Lemma} \label{reslem}
Let $C \subset \ff_n$ be a trigonal curve of Maroni invariant $n$.
Let $K$ be the canonical divisor on $C$ and $H$ the trigonal class. Then
\[ K - (c - a - 2)H = \left.\left(E + \left(\frac{g+n}{2} + 1 - c + a\right)F\right)\right|_C\]
Moreover,
\[ [C] = 3E +\left( \frac{g + 3n}{2} + 1 \right) F.\]
\end{Lemma}
\begin{proof}
Considering a hyperplane in $\pp^{g-1}$ containing the directrix, we see that $\O_{\pp V}(1) = \O_{\ff_n}(E + (\frac{g+n}{2} - 1)F)$, so $K = (E + (\frac{g+n}{2} - 1)F)|_C$. 
Since $H = F|_C$, the first claim follows.
Now suppose $[C] = 3E + \beta F$ for some integer $\beta$. Setting
\begin{align*}
2g - 2 &= \deg K = (E + (\tfrac{g+n}{2} - 1)F) . C \\
&=(E + (\tfrac{g+n}{2} - 1)F).(3E + \beta F) \\
&= -3n + \beta + 3(\tfrac{g + n}{2} - 1),
\end{align*}
we conclude $\beta = \frac{g + 3n}{2} + 1$ as desired.
\end{proof}

To describe the geometry of this linear system, we require a cohomology vanishing statement for non-negative combinations of $E$ and $F$.
\begin{Lemma} \label{cohvan} 
If $m \geq \ell n \geq 0$, then $h^1(\ff_n, \O_{\ff_n}(\ell E + m F)) =0$.
\end{Lemma}
\begin{proof} First consider the sequence
\begin{equation} \label{inE}
0 \rightarrow \O_{\ff_n}((\ell - 1)E + m F) \rightarrow \O_{\ff_n}(\ell E + mF) \rightarrow \O_{E}(m - n \ell) \rightarrow 0.
\end{equation}
Pushing forward by $\pi$, we compute
\begin{align*}
h^0(\ff_n, \O_{\ff_n}(\ell E + mF)) &= h^0(\pp^1, \O_{\pp^1}(m) \otimes \mathrm{Sym}^\ell(\O \oplus \O(-n)))\\
&= (m + 1) + (m - n + 1) + \ldots + (m - \ell n + 1).
\end{align*}
If $m \geq n \ell$, it follows that \eqref{inE} is surjective on global sections. Since $h^1(E, \O_E(m - n \ell)) = 0$, we must have
\[h^1(\ff_n, \O_{\ff_n}(\ell E + m F)) = h^1(\ff_n, \O_{\ff_n}((\ell-1) E + m F)) = \ldots =  h^1(\ff_n, \O_{\ff_n}(mF)),\]
by inducting on $\ell$. Next consider the short exact sequence
\[0 \rightarrow \O_{\ff_n}((m-1) F) \rightarrow \O_{\ff_n}(mF) \rightarrow \O_F \rightarrow 0.\]
A similar argument implies that for $m \geq 0$,
\[h^1(\ff_n, \O_{\ff_n}(m F) = h^1(\ff_n, \O_{\ff_n}((m-1)F)) = \ldots = h^1(\ff_n, \O_{\ff_n}) = 0. \qedhere\]
\end{proof}

One consequence of the above is a description of how general $C$ meet the directrix.
\begin{Lemma} \label{a0}
Let $C \subset \ff_n$ be a general curve of class $3E + (\frac{g+3n}{2} + 1)F$. Then $E \cap C$ consists of $\frac{g - 3n}{2} + 1$ distinct, reduced points.
\end{Lemma}
\begin{proof}
The claim will follows from showing that the restriction map 
\[\O_{\ff_n}(C) \rightarrow \O_{E}(C) \cong \O_{\pp^1}(\tfrac{g - 3n}{2}+1)\]
is surjective on global sections. This follows because the kernel $\O_{\ff_n}(C - E) = \O_{\ff_n}(2E + (\frac{g+3n}{2} + 1)F)$ has no $H^1$ by Lemma \ref{cohvan}.
\end{proof}

We now determine the dimension and base locus of the linear systems of interest.
\begin{Lemma} \label{dim}
Let $C \subset \ff_n$ be a curve of class $3E + (\frac{g+3n}{2} + 1)F$.
For $\alpha \leq \frac{g+n}{2}$, we have
\[h^0(C, \O_C(E+\alpha F))= \begin{cases} 0 & \text{if $\alpha < 0$} \\
\alpha + 1 & \text{if $0 \leq \alpha \leq n-1$} \\
2\alpha -n + 2 & \text{if $\alpha \geq n $}
 \end{cases} \]
 If $0 \leq \alpha \leq n-1$, then $E|_C$ is the base locus of $(E + \alpha F)|_C$. If $\alpha \geq n$, then $(E + \alpha F)|_C$ is base point free.
\end{Lemma}

\begin{proof}
Geometric Riemann-Roch (see e.g. \cite[p.~12]{ACGH}) says $h^0(C, \O_C(D)) - 1$ is the dimension of the span of $D$ in $\pp^{g-1}$ (under the canonical embedding) minus $\deg D -1$. 
Moreover, $p$ is a base point of $D$ if and only if the linear span of $D - p$ is smaller than the linear span of $D$. 

Suppose $n \geq 1$. The intersection $E \cap C$ is a degree $\frac{g - n}{2} + 1 - n$ divisor on a rational normal curve of degree $\frac{g - n}{2} -1$, so its linear span is the largest possible and $h^0(E, \O_C(E)) = 1$. Each fiber $F$ meets some point of $E$. For $\alpha < n$, each time we add points in the restriction of a fiber, the dimension of the linear span of the points increases by $2$ and the number of points increases by $3$. Thus, $h^0$ increases by $1$. It is clear that removing a point in $E \cap C$ decreases the dimension of the linear span, so this is the base locus.

 For $n \leq \alpha  \leq \frac{g+n}{2}$, the span of $(E + \alpha F) \cap C$ contains the linear span of $E$. Adding another triple of points in a fiber only increases the dimension of the span of points by $1$. Thus, $h^0$ increases by $2$ in this range. In this case, our divisor contains enough fibers meeting $E$ that removing a point in $E \cap C$ does not decrease the linear span.
\end{proof}

Let $\Phi$ and $\Phi^\circ$ be as defined in \eqref{phi} and \eqref{phic}.
The following two lemmas give the conditions when $\Phi^\circ$ is non-empty and describe its components.

\begin{Lemma} \label{bign}
If $\alpha = \frac{g + n}{2} + 1 - c +a \geq n$, then $\Phi^\circ \subset \Phi$ is non-empty. If its dimension is positive, $\Phi$ is irreducible.
\end{Lemma}
\begin{proof}
Suppose $D$ contains $m$ points from $E \cap C$. Then we can add up to $\frac{g - n}{2} - m$ pairs of points in fibers without increasing $h^0$ (as well as a single point from each of the remaining fibers). Thus, we can find $(D, D') \in \Phi^\circ$ so long as 
\[\deg D, \deg D' \leq \frac{g-n}{2} + \alpha \quad \Leftrightarrow \quad 1 \leq c - b, b - a,\]
which holds by assumption.

Let $V = H^0(C, K - (c - a - 2)H)$. 
There is a branched cover $\gamma: \Phi \rightarrow \pp V$ sending $(D, D')$ to a section vanishing on $D + D'$.
Assume $\dim V \geq 2$, so $\dim \Phi \geq 1$. Let $U \subset \pp V$ be the divisors in this linear system where each point has multiplicity $1$. In other words, $U$ is the non-tangent hyperplanes to $C$ under the map $C \rightarrow \pp V^\vee$. \textit{Assuming characteristic zero}, $U$ is a dense open subset and 
the Uniform Position Lemma says that the monodromy group of $\{(p, \Lambda) \in C \times U : p \in \Lambda \} \rightarrow U$ is the full symmetric group. It follows that the monodromy of $\gamma^{-1}(U) \subset \Phi \rightarrow U$ is transitive, so $\gamma^{-1}(U)$ is connected. Since $\gamma$ is \'{e}tale over $U$, we see that $\gamma^{-1}(U)$ is smooth, hence irreducible. 
\end{proof}

\begin{prop} \label{smalln}
Suppose $0 \leq \alpha = \frac{g+n}{2} + 1 - c + a \leq n-1$. We have that $\Phi^\circ$ is non-empty if and only if $c - b, b - a \leq 1 + \frac{g-n}{2}$. In this case, 
$\Phi^\circ$ is connected if its dimension is positive. Moreover, if $C$ is general, then $\Phi^\circ$ has
\[ \sum_{\max\{0, b - a - n\} \leq m \leq \frac{g-n}{2} + 1 + b - c} {\frac{g - 3n}{2} + 1 \choose m}\]
irreducible components, correspond to ways of distributing the base points of $(E + \alpha F)|_C$ between $D$ and $D'$.
\end{prop}

\begin{proof}
The condition $c - b, b-a \leq 1 + \frac{g - n}{2}$ is equivalent to $\deg D, \deg D' \geq \alpha$.
If $\deg D < \alpha$, then $D'$ contains a full trigonal fiber, so $h^0(C, D') \geq 2$. The same argument holds with $D$ and $D'$ reversed, so these conditions are necessary for $\Phi^\circ$ to be non-empty. (Notice these conditions were automatically satisfied in Lemma \ref{bign} by its first condition.)

Now suppose $\deg D, \deg D' \geq \alpha$.
Then we can find $(D, D') \in \Phi$ so that both $D$ and $D'$ are subsets of a collection of points of the form $E \cap C$ union $\alpha$ pairs of points on trigonal fibers.
We claim that the later has maximal dimensional span. The span of $E \cap C$ has dimension $\frac{g - n}{2} - n$. Now pick $\alpha$ lines of the ruling on $\ff_n$ not meeting $E$ along $E \cap C$. The linear span of $E \cap C$ union these lines is the same as the linear span of $E \cap C$, union the $\alpha$ points where these lines meet $E$, union the $\alpha$ points where they meet the complimentary rational normal curve.
Because $\alpha \leq n - 1$, the collection of points described on $E$ have maximal span. The points of intersection with the complementary rational curve are also all independent, so all together we have a linear span of dimension $\deg(E \cap C) + 2\alpha - 1$, which is the maximum possible.
Any subset of such a collection of points must also have maximum possible span, and hence has $h^0(C, D) = 1$.

\vspace{-.2in}
\begin{multicols}{2}
The irreducible components of $\Phi^\circ$ correspond to the possible ways of distributing the base points of $E \cap C$ between $D$ and $D'$ so that $h^0(C, D) = h^0(C, D') = 1$. Suppose that we place $m$ base points in $D$ and $m' = \frac{g - 3n}{2} +1 - m$ base points in $D'$. Then in the $\alpha$ fibers there 
are $\deg D - m - \alpha$ that should be partitioned $1$ point in $D$ and $2$ points in $D'$; and $\deg D' - m' - \alpha$ that should be partitioned $1$ point in $D'$ and $2$ points in $D$.
These possibilities for the moving fibers are parametrized by $\Sym^{\deg D - m - \alpha} C \times \Sym^{\deg D' - m' - \alpha} C$, which is irreducible.
\begin{center}
\includegraphics[width=3in]{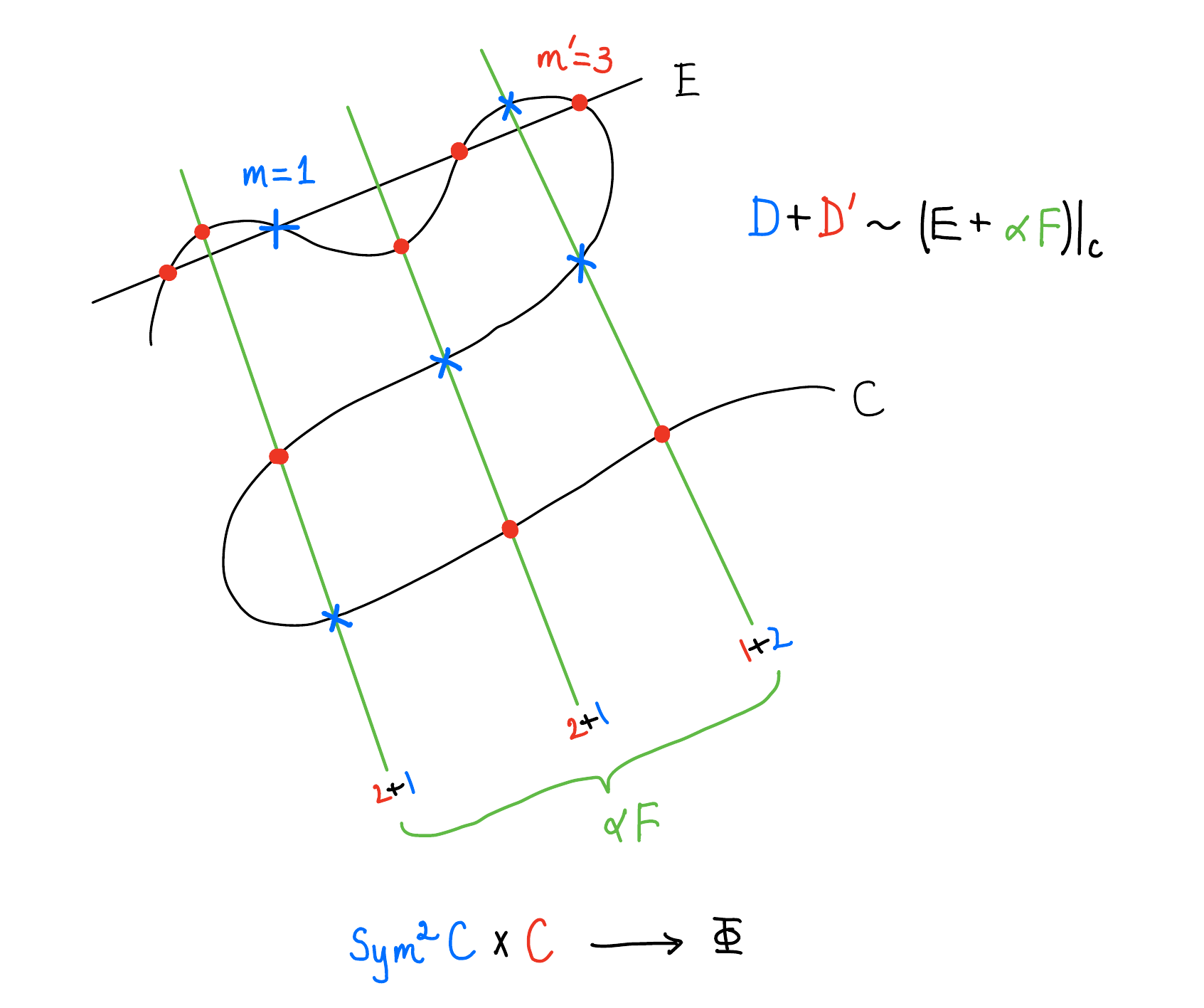}
\end{center}
\end{multicols}

\vspace{-.15in}
A component of $\Phi^\circ$ corresponding to the distribution of base points described above exists so long as $m \geq 0,$ and $\deg D -m-\alpha \geq 0$ and $\deg D' - m' -\alpha \geq 0$, or equivalently
\[\max\{0, b - a - n\} \leq m \leq \frac{g-n}{2} + 1 + b - c.\]
If $C \cap E$ is reduced, it follows that the number of components is the sum of binomial coefficients in the statement of the proposition.

Finally, we prove connectedness when $\alpha > 0$.
If $\max\{0, b - a - n\} = \frac{g - n}{2} + 1 + b - c$ then we have either:
\begin{enumerate}
\item $0 = \frac{g-n}{2} + 1 + b - c$, in which case $\Phi^\circ$ has only one irreducible component, corresponding to all base points living in $D'$. 
\item $b - a - n = \frac{g - n}{2} + 1 + b - c$, which is equivalent to $\alpha = 0$.
\end{enumerate}
Therefore, we may assume $\max\{0, b - a - n\} < \frac{g - n}{2} + 1 + b - c$. This means there are at least two possible values for the number of base points in $D$, so the set of all possible distributions of base points is connected via the following moves: a single point $x \in E \cap C$ moves from $D$ to $D'$ or a single point $x \in E \cap C$ moves from $D'$ to $D$.
Components connected by such a move meet when a moving fiber containing two points of $D$ specializes so as to contain $x$ (resp. a moving fibers containing two points of $D'$ specializes to contain $x$).
\end{proof}

\begin{Example}[Example \ref{ex1} continued: $g = 11, n = 3, (a, b, c) = (-8, -4, -1)$] \label{631}
Here we describe the components of the curve $\Sigma_{(-8,-4,-1)}(C, f)$ when $C \subset \ff_3$ is genus $11$.
We first study $\Phi = \{(D, D') : D + D' \sim K - 6H\} \subset \Sym^3(C) \times \Sym^2(C)$. We have $K - 6H = (E + F)|_C$, so $\alpha = 1$. 
The degree of $E$ under the map to $\pp^{10}$ is $4$ and $E \cap C$ consists of $2$ points, call them $x$ and $y$. Given a point $p$ in $C$, let $p'$ and $p''$ denote points such that $p + p' + p'' \sim F|_C$. The variety $\Phi$ has $4$ components corresponding to each of the following possibilities
\begin{enumerate}
\item[{\color{red} (1)}] $x \in D$ and $y \in D'$: each point on this component corresponds to a point $p \in C$ via $D' = y + p$ and $D = x + p' + p''$.
\item[{\color{blue} (2)}] $y \in D$ and $x \in D'$: this component is a copy of $C$ for the same reason as above.
\item[{\color{black!40!green} (3)}] $x, y \in D$: each point of this component corresponds to a point $p \in C$ via $D = x + y + p$ and $D' = p' + p''$.
\item[(4)] \label{cont} $x, y \in D'$: each point on this component corresponds to a point $q \in \pp^1$ via $D' = x + y$ and $D = f^{-1}(q)$.
\end{enumerate}
We have that $\Phi^\circ$ is the complement of component (4). In fact, component (4) is contracted in the map $\Phi \rightarrow \Sigmabar_{(-8, -4, -1)}$ with image $\O_C$.

\vspace{.05in}
\begin{center}
\begin{tikzpicture}[scale=1.3]
\node at (1, 2.5) {$\Phi$};
  \draw[domain=1.36:2.61,smooth,variable=\x,red] plot ({\x},{5*(\x-2)*(\x-2)});
    \draw[domain=-0.61:0.64,smooth,variable=\x,blue] plot ({\x},{5*(\x)*(\x)});
  \draw[domain = -1:3,smooth,variable=\x,black!40!green] plot ({\x},{-.1*(\x-1)*(\x-1) + .7});
  \draw (.1,1.4) -- (2-.1,1.4);
 
\end{tikzpicture}
\hspace{.2in} 
\begin{tikzpicture}
\draw[->] (0,1) -- (1,1);
\node at (0, 0) {\color{white} h};
\end{tikzpicture}
\hspace{.3in}
\begin{tikzpicture}[scale=1.3]
\node at (1, 2.5) {$\Sigmabar_{(-8,-4,-1)}$};
  \draw[domain=.8:3.1,smooth,variable=\x,red] plot ({\x},{1.4*(\x-2)*(\x-2)});
    \draw[domain=-1.1:1.2,smooth,variable=\x,blue] plot ({\x},{1.4*(\x)*(\x)});
  \draw[domain = -1:3,smooth,variable=\x,black!40!green] plot ({\x},{-.1*(\x-1)*(\x-1) + .7});
  \node[scale=.5] at (1, 1.4) {$\bullet$};
  \node[scale=0.6] at (1.2,1.4) {$\O_C$};
\end{tikzpicture}
\end{center}
\vspace{.05in}
The intersections of components on $\Phi$ are the points $(D, D')$ below, listed left to right:
\begin{align*}
(x+y+x',x+x''), (x+y+x'',x+x'), (y+y'+y'',x+y) \\
 (x+x'+x'',x+y),  (x+y+y', y+y''), (x+y+y'', y + y').
\end{align*}

\end{Example}
\section{Proof of Theorem and Corollary}

\begin{proof}[Proof of Theorem \ref{maint}]
When $c -b \leq 1$ or $b - a \leq 1$, the theorem is proved by Lemma \ref{ez}, so assume we are in the case where $c - b, b - a \geq 2$.
If $\alpha = \frac{g+n}{2} + 1 - c + a \geq n$, then Lemma \ref{bign} says $\Phi^\circ$ is non-empty and irreducible. By Lemma \ref{dim}, we have
\begin{align*}
\dim \Sigma_{(a, b, c)} = \dim \Phi^\circ = h^0(C, \O_C(E + \alpha F)) -1 = 2\alpha - n + 1 = g-u(a, b, c).
\end{align*}
If $1 \leq \alpha < n$, then Proposition \ref{smalln} says $\Phi^\circ$ is non-empty if and only if $c - b, b - a \leq 1 + \frac{g-n}{2}$. In that case, Lemma \ref{dim} shows
\[\dim \Sigma_{(a, b, c)} = \dim \Phi^\circ = h^0(C, \O_C(E+ \alpha F)) - 1 = \alpha =  \frac{g+n}{2} + 1 - c + a.\]
Finally, when $\alpha = 0$, we have $\Phi^\circ = \{(D, D') : D + D' = E \cap C\}$. If $C$ is a general curve of Maroni invariant $n$, Lemma \ref{a0} shows $E \cap C$ consists of $\frac{g - 3n}{2} + 1 = c - a - 2n$ reduced points. Since $\alpha = 0$, we have $\deg D = g + 2 + a + b - 2c = b - a - n$. The points of $\Phi^\circ \cong \Sigma_{(a, b, c)}$ correspond to the ways of choosing which points of $E \cap C$ live in $D$.
\end{proof}

\begin{proof}[Proof of Corollary \ref{cor}]
First assume $c - b \geq 2$ and $b - a \geq 2$.
Suppose $f:C \rightarrow \pp^1$ is a general trigonal curve of genus $g = u(a, b, c) = 2c - 2a - 3$. Since $g$ is odd, the general Maroni invariant is $n = 1$.
By \cite[Theorem 1.2]{L}, $\Sigma_{(a, b, c)} = \Sigmabar_{(a, b, c)}$ is reduced of dimension $0$. Hence,
\[[\Sigmabar_{(a, b, c)}] = \#(\Sigmabar_{(a, b, c)}) \times [\mathrm{pt}] = {c - a - 2 \choose b - a - 1} \times \frac{1}{g!} \theta^g = \frac{1}{u(a, b, c)!} {c - a - 2 \choose b - a - 1} \theta^{u(a, b, c)}.\]
By the universality of the formulas for classes of splitting loci (see \cite[Lemma 5.4]{L}), this is the class of $\Sigmabar_{(a, b, c)}$ whenever it has the expected dimension.

When $c - b \leq 1$ or $b - a \leq 1$, we can run the same argument, considering a general trigonal curve of genus $g = u(a,b,c)$. Here, we know $\Sigma_{(a, b, c)}$ is irreducible so it is a single point and $[\Sigmabar_{(a, b, c)}] = \frac{1}{u(a, b, c)!} \theta^{u(a, b, c)}$. This case is also covered by \cite[Lemma 5.5]{L} (and the corresponding calculation with dual splitting types).
\end{proof}

\end{document}